\documentclass[a4paper,11pt]{article}
\usepackage[utf8]{inputenc}
\usepackage{amsmath,amsthm,amscd,amssymb,eucal,mathrsfs}
\usepackage{amsfonts}
\usepackage{latexsym}
\usepackage{graphicx}
\usepackage{multicol}
\usepackage{dsfont}
\usepackage[all]{xy}
\usepackage{fullpage}

\newtheorem{thm}{Theorem}[section]
\newtheorem{prop}[thm]{Proposition}
\newtheorem{lem}[thm]{Lemma}
\newtheorem{cor}[thm]{Corollary}
\newtheorem{rem}[thm]{Remark}
\newtheorem{dfn}[thm]{Definition}

\newtheorem{exa}[thm]{Example}

\newtheorem*{Satz*}{Satz}

\newcommand{\mathset}[1]{{\left\{#1\right\}}} 
\newcommand{\absolute}[1]{\left\lvert#1\right\rvert}
\newcommand{\norm}[1]{\left\|#1\right\|}

\DeclareMathOperator{\PGL}{PGL}

\DeclareMathOperator{\supp}{supp}

\DeclareMathOperator{\id}{id}

\DeclareMathOperator{\closure}{cl}
\DeclareMathOperator{\Aut}{Aut}
\DeclareMathOperator{\Out}{Out}

\DeclareMathOperator{\Spec}{Spec}

\title{
Generalised diffusion on  moduli spaces of $p$-adic Mumford
curves
}

\author{Patrick Erik Bradley}
\date{\today}

\begin{document}

\maketitle

\begin{abstract}
  A construction of a pseudo-differential operator on non-archimedean local fields invariant under a finite group action is given together with the solution of the corresponding Cauchy problem. This construction is applied to parts of the Gerritzen-Herrlich Teichm\"uller space in order to obtain a self-adjoint operator whose spectrum can decide about certain properties of the reduction graph of the corresponding Mumford curves.
  \end{abstract}

\section{Introduction}

Pseudo-differential operators on local fields, in particular the field $\mathds{Q}_p$ of $p$-adic numbers, have exensively been studied in the literature. The most prominent one being the Vladimirov operator \cite{VVZ1994} which is given by a conjugation with the Fourier transform of a power of the non-archimedean absolute value. More general Pseudo-differential operators are treated in \cite{Zuniga2017}. For an extension of this theory to other non-archimedean local fields, cf.\ \cite{Kochubei2001}. The aim is often to define a diffusion process
with the help of such a pseudo-differential operator. This is possible, if the
corresponding heat kernel has the property of a Markov semi-group, including positivity. Already \cite{Taibleson1975} considered such heat kernels over arbitrary non-archimedean local fields.

\smallskip
A desideratum is an exension of the theory of pseudo-differential operators to $p$-adic manifolds in such a way that one can read off properties of the manifold from the spectrum of the operator, just like in the case of the Laplace operator on graphs or real manifolds. An example in $p$-adic mathematical physics is the Laplace operator of the Bruhat-Tits tree of $\mathds{Q}_p$ which was used to describe $p$-adic string amplitudes \cite{Zabrodin1989}.

\smallskip
Mumford curves can be viewed as a $p$-adic analogon of Riemann surfaces in that they are special $p$-adic algebraic  curves which  allow a finite covering by holed disks. They have a uniformisation theory, and their topological fundamental group is the free group $F_g$ in $g$ generators, where $g$ is the genus of the  curve. This is a so-called Schottky uniformisation.

\smallskip
There is also an adapted version of Teichm\"uller theory for Mumford curves in which the role of the mapping class group is played by the outer automorphism group of the free group $F_g$. The fixed points and multipliers of $g$ generators of the corresponding Schottky group were found by Gerritzen to yield $3g-3$ coordinates for this non-archimedean Teichm\"uller space.

\smallskip
In this article, we are able to construct pseudo-differential operators on a local field which are invariant  under the action of a given finite group.
The corresponding Cauchy problems are formulated and solved by the help of suitable heat kernels.
This construction is applied to the part of the Gerritzen-Herrlich Teichm\"uller space which uniformises the locus in the moduli space of Mumford curves having a fixed reduction type.

\smallskip
The main result concerning Mumford curves is first the construction of a self-adjoint pseudo-differential operator $H_{\mathfrak{G}}$ parametrised by a real number $\lambda$
on the part $\mathcal{F}(\mathfrak{G})$ of a certain fundamental domain 
of the Gerritzen-Herrlich Teichm\"uller space which represents Schottky  groups leading to Mumford curves with stable reduction graph $\mathfrak{G}$.
This part $\mathcal{F}(\mathfrak{G})$ 
depends on a spanning tree $T$ of $\mathfrak{G}$, and the operator $H_{\mathfrak{G}}$ is invariant under the action of the automorphism group of $\mathfrak{G}$.
Secondly, the spectrum of $H_{\mathfrak{G}}$ satisfies
\[
\Spec H_{\mathfrak{G}}\subset p^{\mathds{Z}}-\lambda
\]
if and only if $\mathfrak{G}$ does not contain a mouth-shaped subgraph
having a corner which is a tip of $T$,
such that at most two edges of $\mathfrak{G}$ outside of $T$ are attached to it.

\smallskip
This result shows that it is possible to identify types of reduction graphs of Mumford curves by the spectrum of pseudo-differential operators.
A future task should be the construction of pseudo-differential operators and stochastic processes on general $p$-adic manifolds.

\bigskip
The following section is a brief introduction to Mumford curves as well as moduli and Teich\-m\"uller spaces for these curves. Section 3 is devoted to the  construction of the invariant pseudo-differential operator and its heat kernel
in a general setting.
In Section 4, this construction is applied to the Teichm\"uller space parametrising Schottky groups responsible for a fixed reduction graph, and thus a proof of the the main theorem is obtained.

\section{Mumford curves and their moduli spaces}

\subsection{Mumford curves}

Let $K$ be a non-archimedean local field. The reader may think of $K$ as a finite extension of the $p$-adic number field $\mathds{Q}_p$. Properties of these can be found in \cite{Gouvea}.

\smallskip
A \emph{Mumford curve}  is a projective algebraic curve $X$ defined over $K$ such that
\[
X=\Omega/\Gamma
\]
where $\Gamma\subset\PGL_2(K)$ is a finitely generated free subgroup acting
on the (non-empty)
complement  $\Omega\subset\mathds{P}^1(K)$ of its limit points.
Such a group is called a \emph{Schottky group}.
It is a fact that the number $g$ of generators of $\Gamma$ coincides with the genus of the curve $X$ \cite{GvP1980}.

\smallskip
A large part of the lecture notes \cite{GvP1980} is devoted to
proving the following result by Mumford \cite{Mumford1972} with analytic methods:

\begin{thm}[Mumford]
  There is a one-to-one correspondence between
  \begin{enumerate}
  \item conjugacy classes of Schottky groups in $\PGL_2(K)$
  \item isomorphism classes of
        Mumford curves over $K$
    \end{enumerate}
\end{thm}

\smallskip
If the local field $K$ is large enough, then a Mumford curve has
\emph{split degenerate reduction}, i.e.\ the reduction consists of genus zero curves and has only nodes as singularities. The corresponding reduction graph is then a finite graph whose first Betti number is $g$. The lecture notes \cite{GvP1980} contain more information on the reduction of algebraic curves over  non-archimedean fields. What is relevant for us is the fact that $\Gamma$ acts on an infinite subtree ${T}$ of the Bruhat-Tits tree $\mathcal{T}_K$
associated with the local field $K$, and that
\[
\mathfrak{G}=\tilde{T}/\Gamma
\]
is a \emph{stable} finite graph, i.e.\ a connected graph in which each node is the boundary of at least three edges (a loop-edge counts as two edges in this consideration). The first Betti number of $\mathfrak{G}$ equals $g$.
An introduction to the Bruhat-Tits tree $\mathcal{T}_K$ can be found in
\cite{Serre1980}.

\smallskip
The  possible stable reduction graphs for the case of Mumford curves of genus $g=2$ are depicted in Figure \ref{3graphs}. Notice that the graph $(b)$ can be obtained from either $(a)$ or $(c)$ by contracting an edge. In general, the number of stable reduction graphs is known to be finite. This follows e.g.\ from the dimension
\[
3g-3-\delta
\]
of the locus of stable curves with $\delta$ nodes in the moduli space of stable curves \cite[Ch.\ 2.C]{HM1998}. This bounds the number of nodes. As a node corresponds to an edge in the reduction graph, it follows that the number of edges in a stable graph is bounded. Hence, the number of possible stable reduction graphs is finite.

\begin{figure}
  \[
  \xymatrix{
    *\txt{$\bullet$}\ar@{-}@(ul,dl)\ar@{-}[r]&*\txt{$\bullet$}\ar@{-}@(ur,dr)
    \\
    &\hspace*{-10mm}(a)
  }
  \hspace*{2cm}
  \xymatrix{
    *\txt{$\bullet$}\ar@{-}@(ul,dl)\ar@{-}@(ur,dr)\\
    (b)
  }
  \hspace*{2cm}
     \xymatrix{
    *\txt{$\bullet$}\ar@{-}[r]\ar@{-}@/^/[r]\ar@{-}@/_/[r]&
    *\txt{$\bullet$}\\
   &\hspace*{-10mm}{(c)}
  }
 \]
\caption{The three possible stable reduction graphs for Mumford curves of genus $g=2$.}\label{3graphs}
\end{figure}
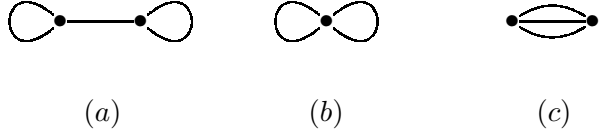

\subsection{Moduli spaces of Mumford curves}

The moduli space $\mathcal{M}_g$  of non-singular projective algebraic curves of genus $g\ge 2$ is a well-studied mathematical object. Its points correspond to isomorphism classes of genus $g$ curves. Let $\mathcal{M}_g(K)$ be the space of $K$-rational points in $\mathcal{M}_g$. These correspond to curves defined over $K$.
The locus $M_g(K)$ of Mumford curves in
$\mathcal{M}_g(K)$ is well-known to be
 an open subspace,
%
%
and
an object of independent interest. In  $p$-adic mathematical physiscs, $M_g$ was used in the context of string theory \cite{CMZ1989}.

\smallskip
The space $M_g$ is the disjoint union of
finitely many parts $M_g(\mathfrak{G})$ with fixed stable reduction graph $\mathfrak{G}$. For example, in the case of genus $g=2$, there are three such parts, each belonging to a graph depicted in Figure \ref{3graphs}.

\smallskip
In order to ensure that each part $M_g(\mathfrak{G})$ is non-empty, the non-archimedean field $K$ must be sufficiently large. The reason is that
the tree $\tilde{T}
$ in the universal covering
\begin{align}\label{unigraph}
\tilde{T}\to\mathfrak{G}=\tilde{T}/\Gamma
\end{align}
needs to be embedded into the Bruhat-Tits tree $\mathcal{T}_K$ in order to obtain a universal covering
\[
\Omega\to X = \Omega/\Gamma
\]
with $X$ a Mumford curve having stable reduction graph $\mathfrak{G}$.
The map (\ref{unigraph}) is obtained with the help of a spanning tree $T$ of $\mathfrak{G}$ as a fundamental domain for the action of the free group $\Gamma$.
Now, if $K$ is not large enough, then $T$ can have more branching than the $\mathcal{T}_K$: we remind that the number of edges attached to a vertex in $\mathcal{T}_K$ is
\[
q+1=p^f+1
\]
where $f$ is the ramification index of the field extension $K/\mathds{Q}_p$ \cite{Serre1980}.
%
%
In that case, $T$ cannot be embedded into $\mathcal{T}_K$, and so the space
$M_g(\mathfrak{G})$ is empty.

\smallskip
We will restrict to unramified extensions of $\mathds{Q}_p$, i.e.\ to field extensions $K/\mathds{Q}_p$ whose degree equals the ramification index $f$. In this case, the prime $p$ is also a uniformiser for $K$ \cite[Ch.\ II]{Neukirch}. This means that every element  $x\in K$ is of the form
\[
x=\sum\limits_{\nu=-m}^\infty x_\nu p^\nu
\]
with $x_\nu\in\mathcal{R}$ for some choice of $p$-adic digits in $\mathds{D}_K$, where
\[
\mathds{D}_K=\mathset{ z\in K\colon\absolute{z}_K\le 1}
\]
is the unit disk.

\subsection{Teichm\"uller spaces for Mumford curves}

The  \emph{Gerritzen-Herrlich Teichm\"uller space} is defined as
  \[
\mathcal{T}_g(K)=\mathset{F_g\stackrel{\text{repr.}}{\longrightarrow}\PGL_2(K)\colon\text{faithful discrete}}/PGL_2(K)
\]
where
$g\ge 2$, and $\text{repr.}$ stands for ``representation''.
The image of a representation in $\mathcal{T}_g(K)$ is a {Schottky group}. Hence, it  consists of hyperbolic M\"obius transformations.

\smallskip
It is known that $\mathcal{T}_g(K)\subset K^{3g-3}$ is an open analytic polyhedron \cite{Gerritzen1981}. Herrlich considered spaces of projective-linear representations of more general groups \cite{Herrlich1987}.

\smallskip
There is an action of the outer automorphism group 
$\Out(F_g)
$
of the free group $F_g$ in $g$ generators on $\mathcal{T}_g(K)$
given by bi-analytic maps.
The quotient space

\[
M_g(K)=\mathcal{T}_g(K)/\Out(F_g)
\]
is the moduli space of Mumford curves of genus $g$ \cite{Gerritzen1981}.

  \begin{thm}[Gerritzen]\label{fundamental}
There is
    a fundamental domain $\mathcal{F}\subset \mathds{D}_K^{3g-3}$ which is the disjoint union of pieces $\mathcal{F}(\mathfrak{G})$ with stable reduction graph $\mathfrak{G}$ such that
    \[
M_g=\bigcup\limits_{\mathfrak{G}}\mathcal{F}(\mathfrak{G})/\Aut(\mathfrak{G})
\]
  \end{thm}

  \begin{proof}
Cf.\ \cite[Satz 1]{Gerritzen1981}.
  \end{proof}

  For the proof of Theorem \ref{fundamental},
  Gerritzen constructs $\mathcal{F}$ from a spanning tree $T$ of $\mathfrak{G}$. Each of the $g$ edges in $\mathfrak{G}$ outside $T$ are part of a basis of the fundamental group of $\mathfrak{G}$ (lasso loops) which corresponds to $g$ generators of a Schottky group inside $\PGL_2(K)$.
  The fixed points and multipliers of these generators are the coordinates of the embedding $\mathcal{F}(\mathfrak{G})$ into $\mathds{D}_K^{3g-3}$ after
  bringing them into the form
  \[
x=(0,1,t_1;\infty,y_2,t_2;x_3,y_3,t_3,\dots,x_g,y_g,t_g)
\]
with
\[
0<\absolute{t_i}_K<1,\quad y_2,x_i,y_g\in\mathds{D}_K,\quad i=3,\dots,g
\]
with the help of a suitable M\"obius transformation.

\smallskip
We will also write
\[
\dot{\mathds{D}}_K^-=\mathset{x\in K\colon 0<\absolute{x}_K<1}
\]
for the maximal punctured disk of radius less than $1$.

\section{A general setting}\label{extension}

 Let $X\subset \mathds{D}_K^N$ be an open compact set, and
$G$ a finite group acting on $X$ via bi-analytic maps.
 We will extend $\sigma\in G$ to all of $K^N$ in three steps.

 \smallskip
 The norm on $K^N$ will be denoted as
 \[
\norm{x}_K=\max\mathset{\absolute{x_1}_K,\dots,\absolute{x_N}_K}
\]
where
$x=(x_1,\dots,x_N)\in K^N$.

\smallskip
In the case of the Gerritzen-Herrlich Teichm\"uller space,
we have
\[
N= 3g-3
\]
Also, in this case,
first extend $\sigma\in\Aut\mathfrak{G}$ to $\partial\mathcal{F}(\mathfrak{G})$
in the natural way. Then we can set 
     \begin{align*}
       X &= \closure_{K^N}\mathcal{F}(\mathfrak{G})\\
       G&=\Aut(\mathfrak{G})
     \end{align*}
After this,  extend $\sigma\in G$ to $K^N\setminus\closure_{K^N}\mathcal{F}(\mathfrak{G})$ in the three steps to follow.

\begin{enumerate}
  \item Let $\sigma=\id$ on $\mathds{D}_K^N\setminus X$.
  \item Let $\rho>>0$. Define
    \begin{align*}
B_\rho(0)&=\mathset{x\in\mathds{K}^N\colon \norm{x}_K\le p^\rho}
    \end{align*}
    and let
    \begin{align*}
s_\rho\colon&B_\rho(0)\to\mathds{D}_K^N,\;x\mapsto p^\rho x
    \end{align*}
be a rescaling operator.    Then let  on $B_\rho(0)\setminus\mathds{D}_K^N$
    \[
\sigma\colon x\mapsto s_\rho^{-1}\sigma s_\rho(x)
\]
be the extension.
\item Let $\sigma=\id$ on $K^N\setminus B_\rho(0)$.
  \end{enumerate}

Notice that we have used the symbol $\sigma$ for the original map on $X$ as well as for its extension to $K^N$. This should not be a cause for confusion, as in the following we will use only the extension.

\subsection{Twisted heat kernel}
   Let $G$ be a finite group acting on $K^N$ (e.g.\ whose  action is obtained as in the previous section) and let $\sigma\in G$.

  \begin{dfn}
A function    $f\colon K^N\to \mathds{R}$ is called \emph{$\sigma$-radial}, if
    \[
\norm{\sigma x}_K=\norm{\sigma y}_K\;\Rightarrow\;f(x)=f(y)
\]
  \end{dfn}

Observe that an $\id$-radial function is the same as a radial function.

    \begin{dfn}
        Let $f\colon K^N\to\mathds{R}$ be $\sigma$-radial.  
      $f$ is called \emph{$\sigma$-increasing}, if
      \[
\norm{\sigma x}_K\le\norm{\sigma y}_K\;\Rightarrow\; f(x)\le f(y)
      \]
      \end{dfn}

    Observe that an $\id$-increasing function
    is the same as an increasing (or non-decreasing) function.

    \bigskip
    Let $f\colon K^N\to\mathds{R}$, and $\sigma\in G$.
    Then we can define a new function
  \[
f_\sigma\colon K^N\to\mathds{R},\;x\mapsto f(\sigma x)
\]

  \begin{lem}
    The following holds true:
      \begin{enumerate}
      \item If $f$ is radial, then $f_\sigma$ is $\sigma$-radial.
      \item If $f$ is also increasing, then $f_\sigma$ is $\sigma$-increasing.
        \end{enumerate}
  \end{lem}

      \begin{proof}
This is immediate.
        \end{proof}

Now, let $f\colon K^N\to\mathds{R}$ be a function satisfying the following conditions:
  \begin{enumerate}
  \item $f$ is radial
  \item $f(x)=\lambda$ for $x\in\mathds{D}_K^N$
  \item $f$ is increasing
  \item There are $A_1,A_2,\gamma_1,\gamma_2>0$ s.t.
    \[
A_1\norm{x}_K^{\gamma_1}\le f(x)\le A_2\norm{x}_K^{\gamma_2}\qquad (\norm{x}_K>>1)
\]
    \end{enumerate}

  We will write $dx$ for the Haar measure on $K^N$ which is normalised such
  that the measure of the unit ball is $1$.
  The function $\chi\colon K\to\mathds{C}^\times$ will denote
  the {standard additive character} on $K^N$.
And $S(K^N)$ means the space of test functions. These consist of complex-valued functions on $K^N$ which are locally constant with compact support.
    
\medskip
  Let
  \begin{align*}
    F\colon S(K^N)&\to S(K^N),\;f\mapsto Ff
  \end{align*}
  with
  \begin{align*}
    f(y)&=\int\limits_{K^N}\chi(y\cdot x) f(x)\,dx
    &(y\in K^N)
\end{align*}
  be the Fourier transform.
  We will sometimes write
  \[
F_{x\to y}
\]
instead of $F$ in order to make clear the variables on both sides of the transformation map.

        \begin{dfn}
        Let $\sigma\in G$. The expression
    \[
Z_\sigma(x,t)=F^{-1} e^{-t(f_\sigma(\cdot)-\lambda)}(x)
\]
with $x\in K^N$
is a \emph{twisted heat kernel}.
    \end{dfn}

        We are interested in the following properties of a function
        \[
Z\colon K^N\times \mathds{R}_{\ge 0}\to \mathds{C}
\]
Namely,
\begin{enumerate}
    \item $\supp(Z(\cdot,t))\subset\mathds{D}_K^N$ for $t>0$
            \item $Z(x,t)\in C(K^N,\mathds{R})\cap L_1(K^N)\cap L_2(K^N)$ for $t>0$.
    \item $\int\limits_{K^N}Z(x,t)\,dx =1$ for $t>0$
    \item $\lim\limits_{t\to 0^+}Z(x,t)*\psi(x)=\psi(x)$, $\psi\in S(K^N)$            \item $Z(x,t+t')=Z_\sigma(x,t)*Z_\sigma(x,t')$ for $t,t'>0$
              \end{enumerate}

          \begin{rem}
          If one also adds the following property:
          
          \begin{itemize}
                 \item[6.] $Z(x,t)\ge 0$ for $t>0$
          \end{itemize}
          then one can associate to $Z$ a sochastic process.
However, as in this article we are not interested in stochastic processes, but more in characterising parts of the moduli space via the spectrum of an associated self-adjoint operator, we will not require the last property, i.e.\ positivity of the kernel $Z$. 
          \end{rem}

  \begin{lem}\label{ZsigmaProperties}
    Let $\sigma\in G$. Then  $Z_\sigma(x,t)$ has the  properties 1.\ to 5.
    If $\sigma=\id$, then also property 6.\ holds true.
  \end{lem}

  \begin{proof}
    We follow the outline of the proof in \cite{CR2015}, where the properties were proven for $Z_{\id}(x,t)$ in the case of $K=\mathds{Q}_p$.

    \smallskip\noindent
    1. We have
    \[
    Z_\sigma(x,t)=\int\limits_{\norm{\xi}_K\le 1}\chi(-x\cdot\xi)\,d\xi
    +e^{\lambda t}\int\limits_{\norm{\xi}_K>1}\chi(-x\cdot \xi)e^{-tf_\sigma(\xi)}\,d\xi
    \]
    The first integral is the Fourier transform of the indicator function of the unit ball, which again is  the same indicator function. For the second integral, note that if $\norm{x}_K>1$, then
    \[
\int\limits_{\norm{\xi}_K=p^\nu}\chi(-x\cdot \xi)\,d\xi=0
\]
for $\nu>0$. Hence,
\[
Z_\sigma(x,t)=0
\]
for $\norm{x}_K>1$.
    
    \smallskip\noindent
    2. As for $\rho\ge 1$ and fixed $t>0$, we have
    \[
e^{-t f_\sigma}\in L_\rho(K^N)
\]
(the proof of \cite[Lemma 3.2]{CR2015} carries over), this  assertion follows immediately.

\smallskip\noindent
3. This follows from 2.\ by the inversion formula for the Fourier transform.

  \smallskip\noindent
  4. From 3.\ it follows that
  \[
  \int\limits_{K^N}Z_\sigma(x-\xi)\phi(\xi)\,d\xi-\phi(\xi)
  =\int\limits_{K^N} Z_\sigma(x-\xi)(\phi(\xi)-\phi(x))\,d\xi
  \]
  The  proof of \cite[Lemma 3.6 (iii)]{CR2015} now carries over to this case.

  \smallskip\noindent
  5. By 2.\ this property is equivalent to
  \[
FZ_\sigma(x,t)\cdot  FZ_\sigma(x,t')=FZ_\sigma(x,t+t')
\]
As
\[
FZ_\sigma(\cdot,\tau)=e^{-\tau (f_\sigma-\lambda)}
\]
the property follows.

\smallskip\noindent
6. If $\sigma=\id$, then $Z_\sigma$ coincides with $Z$ from \cite{CR2015} for $K=\mathds{Q}_p$. The proof of property 6.\ carries over to the case of our more general $K$.
    \end{proof}

  We now define the following linear operators:
  
  \begin{align*}
    J_\sigma\psi(x)&=F^{-1}_{\xi\to x}(f_\sigma(\xi)F_{y\to \xi}\psi(y))\\
    H_\sigma\psi&=J_\sigma\psi-\lambda\psi
  \end{align*}

  \begin{lem}
    It holds true that 
      \begin{enumerate}
      \item $J_\sigma\colon S(K^N)\to S(K^N)$ is a homeomorphism.
      \item $H_\sigma\colon S(K^N)\to S(K^N)$ is continuous and self-adjoint.
        \end{enumerate}
    \end{lem}

  \begin{proof}
    1. The linear operator $J_\sigma$ is a composition of homeomorphisms $S(K^N)\to S(K^N)$. Hence, it also is a homeomorphism.

    \smallskip\noindent
    2. \emph{Continuity}. This is an immediate consequence of $1$.

    \smallskip
    \emph{Self-adjointness.} $H_\sigma$ is obtained by conjugating a real-valued function with the  Fourier transform. This shows that it is self-adjoint.
  \end{proof}

  \begin{rem}
Observe that $J_{\id}$ and $H_{\id}$ are the operators $J$ and $H$ from \cite{CR2015}.
  \end{rem}
  
        \begin{thm}\label{CPsigma}
      Let $\sigma\in G$. Then the Cauchy problem
      \begin{align*}
        \frac{\partial u_\sigma(x,t)}{\partial t}&+H_\sigma u_\sigma(x,t)=0\\
        u_\sigma(x,0)&=\psi(x)\in S(K^N),&x\in K^N,\;t>0
      \end{align*}
      is solved by
      \[
u_\sigma(x,t)=Z_\sigma(x,t)*\psi(x)
\]
This function is continuously differentiable in time for $t\ge 0$ (pointwise).
    \end{thm}

        \begin{proof}
          As in \cite[Claim 3.13 (i)]{CR2015}, one proves first that
          \[
\frac{\partial}{\partial t} e^{-t(f_\sigma-\lambda)}=-(f_\sigma-\lambda)e^{-t(f_\sigma-\lambda)}
\]
and for $t,t_0\in [0,T]$ that
\begin{align}\label{bigIneq}
\absolute{\frac {e^{-t(f_\sigma-\lambda)}-e^{-t_0(f_\sigma-\lambda)}}{t-t_0}
  +(f_\sigma-\lambda)e^{-t_0(f_\sigma-\lambda)}}\le C(t)\norm{\cdot}_{K,\sigma}^{\gamma_2}
\end{align}
The Dominated Convergence Theorem then shows that $u_\sigma$ is continuously differentiable in $t\ge 0$. The Parseval-Steklov Theorem now shows that $u_\sigma$ solves the Cauchy problem.
        \end{proof}
        
\subsection{Invariant heat kernel}

In the previous section, we defined twisted heat kernels $Z_\sigma(x,t)$ for each $\sigma\in G$, where $G$ is a group acting on $K^N$. Now we put these together to define a heat kernel for all of $G$.

    \begin{dfn}
      \[
      Z_G(x,t)=\frac{1}{\absolute{G}}\bigstar_{\sigma\in G}Z_\sigma(x,t)
      \]
      is the \emph{invariant heat kernel} for the action of $G$ on $X$.
    \end{dfn}

    Let
      \[
H_G=\frac{1}{\absolute{G}}\sum\limits_{\sigma\in G}H_\sigma
\]


\begin{lem}
The invariant heat kernel $Z_G$ satisfies the properties 1.\ to 5. If $G$ is the trivial group, then also property 6.\ holds true.
\end{lem}

\begin{proof}
1.\ This follows from the fact that the convolution of functions supported inside the unit ball is also supported inside the unit ball.

\smallskip\noindent
2. This holds true, as $Z_G$ is the convolution of functions with that property.

\smallskip\noindent
3. This follows in the same way as for $Z_\sigma$ with $\sigma\in G$ (cf.\ Lemma \ref{ZsigmaProperties}).

\smallskip\noindent
4. This follows by successive application of property 4.\ for each individual $Z_\sigma$.

\smallskip\noindent
5. This follows by the same method as used for property 4.

\smallskip\noindent
6. If $G$ is trivial, then $Z_Z$ coincides with $Z$ from \cite{CR2015}, if $K=\mathds{Q}_p$.
The proof of property 6.\ carries over to our more general field $K$.
\end{proof}

  \begin{thm}
    The Cauchy problem
    \begin{align*}
      \frac{\partial u(x,t)}{\partial t}&+H_G\;\! u(x,t)=0\\
      u(x,0)&=\psi(x)\in S(K^N),&x\in K^N,\;t>0
    \end{align*}
    is solved by
    \[
u(x,t)=Z_G(x,t)*\psi(x)
\]
This function is continuously differentiable in time for $t\ge 0$ (pointwise).
  \end{thm}

  \begin{proof}
    The proof is analogous to Theorem \ref{CPsigma}, the case of the individual $H_\sigma$ with $\sigma \in G$.
    Notice that now there is a product of exponential functions whose time derivative is a factor times that product. The inequality corresponding to (\ref{bigIneq}) is
    \begin{align}
      \absolute{
        \frac{\exp\left(-t\left(\sum\limits_{\sigma\in G}f_\sigma-\lambda\right)\right)
          -
          \exp\left(-t_0\left(\sum\limits_{\sigma\in G}f_\sigma-\lambda\right)\right)
        }{t-t_0}
      }
&      \le C(t)\sum\limits_{\sigma\in G}\norm{\cdot}_{K,\sigma}^{\gamma_2}
    \end{align}
    So, again the Dominated Convergence Theorem yields the continously differentiability of $u$, and the Parseval-Steklov Theorem yields that $u$ is a solution of the Cauchy problem.
  \end{proof}

  \begin{lem}\label{commute}
    The $H_\sigma$ with $\sigma\in G$ commute, i.e.
    \[
H_\sigma H_\tau = H_\tau H_\sigma
\]
for $\sigma,\tau\in G$.
  \end{lem}

  \begin{proof}
    It holds true that
    \begin{align*}
      H_\sigma H_\tau&=\left(F^{-1} f_\sigma F-\lambda \id\right)\left(F^{-1}f_\tau F-\lambda\id\right)\\
      &=F^{-1}f_\sigma f_\tau F-\lambda F^{-1}\left(f_\sigma+f_\tau\right)+\lambda^2\id\\
     & =F^{-1}f_\tau f_\sigma F-\lambda F^{-1}\left(f_\tau+f_\sigma\right)+\lambda^2\id\\
      &=H_\tau H_\sigma
    \end{align*}
This proves the assertion.
  \end{proof}

  \begin{cor}
    The $H_\sigma$ with $\sigma\in G$ are simultaneously diagonalisable.
    Each eigenvalue of $H_G$ is the average of the  eigenvalues
    corresponding to a fixed eigenfunction of $H_\sigma$, where $\sigma$ varies over $G$.
  \end{cor}

  \begin{proof}
    This is an immediate consequence of Lemma \ref{commute}.
  \end{proof}
  
  Let $\gamma\le 0$ be an integer,
  and let $k  \in \mathcal{R}^N$ with $\mathcal{R}$ a system of $\pi$-adic digits. Further, let  $b\in (K/O_K)^N$.
  
    \begin{prop}\label{ev+ef}
      Let $\sigma\in G$. Then:
   \begin{enumerate}
\item    There is a complete orthonormal basis  
  $\omega_{\gamma bk}$  of $L_0^2(O_K^N)$ consisting of eigenfunctions of $J_\sigma$.
  \item
For fixed integer  $\gamma\le 0$ there are only finitely many eigenfunctions
$\omega_{\gamma bk}$ satisfying
\[
J_\sigma\, \omega_{\gamma bk} = \lambda_{\gamma_\sigma bk}\, \omega_{\gamma b k}
\]
where
\[
\lambda_{\gamma_\sigma bk} = \tilde{f}(p^{1-\gamma_\sigma}),\quad \gamma_\sigma\in\mathds{Z}
\]
where
\[
\gamma_{\id}=\gamma
\]
and $\tilde{f}$ is the function on $p^{\mathds{Z}}$ induced by the radial
function $f$.
   \end{enumerate}
   \end{prop}

    \begin{proof}
      The proofs of Lemma 12.3, Remark 12.4 in \cite{XKZ} carry over to this situation. Notice that
      \[
      f_\sigma(p^\gamma(-p^{-1}k+\eta)
      =\tilde{f}_\sigma\left(\norm{p^\gamma(p^{-1}k+\eta)}_K\right)
=\tilde{f}\left(\norm{\sigma\left( p^\gamma(p^{-1}k+\eta)\right)}_K\right)
      =\tilde{f}(p^{1-\gamma_\sigma})
\]
where $\tilde{f}_\sigma$ is the function on $p^{\mathds{Z}}$ induced by $f_\sigma$,
$\eta\in O_K^N$, and $k\in\mathcal{R}^N$, where $\mathcal{R}$ is a system of $p$-adic digits for $K$.
    \end{proof}

    \section{Generalised  diffusion on $M_g(\mathfrak{G})$}

    Let $\mathfrak{G}$ be a stable graph of genus $g\ge 2$.
    A \emph{geometric basis} $W=\mathset{w_1,\dots,w_g}$ of the free group
    $F_g$ corresponds to a set of lasso loops starting in a fixed vertex $P$ of $\mathfrak{G}$.
  An element $\sigma\in G$ maps this to another geometric basis which
  corresponds to a set of lasso loops also starting in the vertex $P$.
  Define
  \[
  w_{-i}=w_i^{-1}
  \]
for $w_i\in W$.

    \begin{thm}[Gerritzen]
      It holds true that
      \[
\sigma(w_i)=w_{t_1}\cdots w_{t_s}w_{j_1}\cdots w_{j_r}w_{t_s}^{-1}\cdots w_{t_1}^{-1}
\]
with $t_i,j_k\in\mathset{\pm 1,\dots,\pm g}$ of pairwise distinct (Archimedean) absolute value.

    \end{thm}
\begin{proof}
  Cf.\ the proof of
  \cite[Satz 4]{Gerritzen1981}.
\end{proof}

\subsection{Generalised diffusion on the parts of $M_2$}\label{g=2}

The moduli space $M_2$ consists of three parts: $M_2(\mathfrak{G}_a)$, $M_2(\mathfrak{G}_b)$ and $M_2(\mathfrak{G}_c)$ which correspond to the three graphs of Figure \ref{3graphs} with $\mathfrak{G}_i$ corresponding to graph $(i)$ in
the figure for $i\in\mathset{a,b,c}$.

\smallskip
We will now study the three parts of $M_2$ individually.

\bigskip\noindent
    {\bf (a)} We have

    \[
    \Aut\mathfrak{G}_a\cong\mathds{Z}/2\mathds{Z}=\langle\sigma\rangle
    \]
    where $\sigma$ interchanges the two loop-edges. Hence,
    \[
\Aut\mathfrak{G}_2\cdot W^{\pm}=W^{\pm}
    \]
    where $W$ is the geometric basis of $F_2$ consisting of the two lasso loops starting in a fixed vertex. By Proposition \ref{niceAut} below, it follows that
    \[
\Spec H_{\mathfrak{G}_a}\subset p^{\mathds{Z}}-\lambda
    \]
    in this case.

    \bigskip\noindent
        {\bf (b)} The graph $\mathfrak{G}_b$ has the same automorphism group
        as $\mathfrak{G}_a$, and the action is the same. Hence,
        \[
\Spec H_{\mathfrak{G}_b}\subset p^{\mathds{Z}}-\lambda
\]
also in this case.

\bigskip\noindent
    {\bf (c)} An automorphism of the graph $\mathfrak{G}_c$ can permute the edges or interchange the two vertices. Hence,
    \[
\Aut\mathfrak{G}_c\cong S_3\times \mathds{Z}/2\mathds{Z}
\]
Let $W=\mathset{w_1,w_2}$ be the following geometric basis of $F_2$:
\[
w_1= \quad\!\!\!
\xymatrix{
    *\txt{$\bullet$}\ar@/^/[r]&*\txt{$\bullet$}\ar[l]
}
\qquad
w_2=\quad\!\!\!
  \xymatrix{
    *\txt{$\bullet$}\ar@/_/[r]
&    *\txt{$\bullet$}\ar[l]
  }
\]
We consider the following automorphism $\sigma\in\Aut\mathfrak{G}_c$ whose effect on $W^{\pm}$ is given as follows:
\begin{align*}
  w_1&\mapsto w_1\\
  w_2&\mapsto\quad\!\!\!
    \xymatrix{
      *\txt{$\bullet$}\ar@/^/[r]&
      *\txt{$\bullet$}\ar@/^/[l]
  }
  \quad\!\!\!=w_2^{-1}w_1\notin W^{\pm}
\end{align*}

The task is now to find explicit points in $x\in\mathcal{F}(\mathfrak{G}_c)$
with
\[
\norm{\sigma x}_K<\norm{x}_K
\]
such that we arrive at two different eigenvalues $\lambda_{\id},\lambda_{\sigma}$ corresponding to the same eigenfunction of $H_{\id}$ (or $H_\sigma$, which does not matter). This then shows that
\[
\Spec H_{\mathfrak{G}_2}\not\subset p^{\mathds{Z}}-\lambda
\]
in this case.

\smallskip
We will write the coordinate vectors of $\mathcal{F}(\mathfrak{G}_2)$ as
\[
x=(0,1,t_1;\infty,y_2,t_2)
\]
with
\[
0<\absolute{t_i}_K<1\quad(i=1,2),\quad \absolute{y_2}\le 1
\]
So, after obtaining the action on such a tuple, we need to bring the result back to this form with the help of a suitable M\"obius transformation.

\smallskip
Matrices corresponding to the hyperbolic transformations $w_1$ and $w_2$ are
given as
\[
w_1\leftrightarrow\begin{pmatrix}
  -t_1&0\\
  1-t_1&-1
\end{pmatrix}
\quad\text{and}\quad
w_2\leftrightarrow
\begin{pmatrix}
  1&(t_2-1)y_2\\
  0&t_2
  \end{pmatrix}
\]
and we have
\[
w=w_2^{-1}w_1\leftrightarrow
\begin{pmatrix}
  y_2(1-t_1)(1-t_2)-t_1t_2&y_2(t_2-1)\\
  1-t_1&-1
  \end{pmatrix}
\]
We will pick $y_2=-1$, which specialises the latter matrix to
\[
w\leftrightarrow
\begin{pmatrix}
  t_1+t_2-1&t_2-1\\
  1-t_1&-1
\end{pmatrix}
\]
This leads to the following identity:
\begin{align*}
\sigma(0,1,t_1;\infty,-1,t_2)=(0,1,t_1;\infty,\eta,t)
\end{align*}
with
\[
t=t_1t_2
\]
We have
\[
\norm{(t_1,-1,t_2)}_k=1
\]
and would like to have
\[
\norm{(t_1,\eta,t)}_K<1
\]
This is the case, if we find $t_1,t_2$ such that
\[
\absolute{\eta}_K<1
\]
Given a  M\"obius transformation 
\[
z\mapsto \frac{az+b}{cz+d}
\]
with $c\neq 0$,
a point $z\in K$ is a fixed point, if and only if
\begin{align}\label{fpz1z2}
z^2+\frac{d-a}{c}z-\frac{b}{c}=0
\end{align}
In the case of the M\"obius transformation $w$ given by $y_2=-1$, it follows that its fixed points $z_1,z_2$ must satisfy
\begin{align}\label{lineq0}
  z_1 z_2=\frac{b}{c}&=\frac{1-t_2}{1-t_1}\\
  z_1+z_2=\frac{d-a}{c}&=\frac{t_1+t_2}{t_1-1}
\end{align}
This is equivalent  to the system of linear equations in $t_1,t_2$:
\begin{align}
  z_1z_2t_1-t_2&=z_1z_2-1\label{lineq}\\\nonumber
  (z_1+z_2-1)t_1-t_2&=z_1+z_2
\end{align}
which has a unique solution unless
\begin{align}\label{criticalCurve}
z_1z_2=z_1+z_2-1
\end{align}
in which case there are infinitely many solutions.

\smallskip
Now, we want to map between the tuples
\[
(z_1,z_2,1,0)\to(0,\infty,1,\eta)
\]
via a suitable M\"obius transformation. The first three elements of each tuple uniquely determine a M\"obius transformation. This can be done with the transformation
\begin{align}\label{MoebiusBeta}
\beta(z)=\frac{(z_2-1)z+z_1(1-z_2)}{(z_1-1)z+z_2(1-z_1)}
\end{align}
Applying this M\"obius transformation $\beta$ to $0$ yields
\[
\eta=\frac{z_1(z_2-1)}{z_2(z_1-1)}
\]

We have
\[
\absolute{\eta}_K<1
\]
if and only if
\begin{align}\label{condz1z2}
\absolute{\frac{z_1}{z_2}}_K<\absolute{\frac{z_1-1}{z_2-1}}_K
\end{align}

Hence, it is possible to find $t_1,t_2$ such that $\absolute{\eta}_K<1$.
The question is, whether there exist such $t_1,t_2$ for which
\[
0<\absolute{t_i}_K<1
\]
with $i=1,2$.

\begin{lem}[Case  $y_2=-1$]
  There exists a $1$-parameter family
  \[
  F_\epsilon\colon\mathds{D}_K^-\to
  \left(\dot{\mathds{D}}_K^-\right)^3,\;
  \epsilon\mapsto\mathset{(t_1,t_2,\eta)\colon t_1=t_1(\epsilon),t_2=t_2(\epsilon),
    \eta=\eta(\epsilon)}
\]
such that
\[
\sigma(t_1,-1,t_2)=(t_1,\eta,t_1t_2)
\]
for generic $\epsilon\in\mathds{D}_K^-$.
\end{lem}

\begin{proof}
  In the case of (\ref{criticalCurve}), the system (\ref{lineq}) reduces to its
  first equation. As $z_1z_2\neq 0$, we can write the solution as
  \[
  \begin{pmatrix}t_1\\t_2\end{pmatrix}
    =\begin{pmatrix}\frac{z_1z_2-1}{z_1z_2}\\0\end{pmatrix}
    +s\begin{pmatrix}\frac{1}{z_1z_2}\\1\end{pmatrix}
    \]
    with $s\in K$. In order to have $(t_1,t_2)\in\left(\dot{\mathds{D}}^-_K\right)^2$, it has to hold true that
    \[
\absolute{z_1+z_2}_K<1\quad\text{and}\quad s\in\dot{\mathds{D}}_K^-
\]
Intersecting this with the curve given by (\ref{criticalCurve}), this means that
 we have the equations
\begin{align*}
  z_1+z_2&=\epsilon\\
  z_1z_2&=1+\epsilon
\end{align*}
with $\absolute{\epsilon}_K<1$. These equations are equivalent to the quadratic equation
\begin{align}\label{quadratic}
z_1^2-\epsilon z_1+(1+\epsilon)=0
\end{align}
parametrised by $\epsilon$. For generic $\epsilon\in\mathds{D}^-$,
there are two distinct solutions of (\ref{quadratic}). This proves the assertion.
\end{proof}

The eigenvalues of the operator $J_\tau$ with $\tau\in\Aut\mathfrak{G}_c$
for a fixed eigenfunction are given as
\[
\lambda_\tau=\norm{p^\rho\tau(t_1,-1,t_2)}_K
\]
for $\rho>>0$ \cite[Proof of Lemma 12.3]{XKZ} (cf.\ also Proposition \ref{ev+ef}). In our case, we have
\[
\lambda_{\sigma}=\norm{p^\rho\sigma(t_1-1,t_2)}_K=\norm{p^\rho(t_1,\eta,t_1t_2)}_K
<\norm{p^\rho(t_1,-1,t_2)}_K=\lambda_{\id}
\]
From this, it follows that
\[
\Spec H_{\mathfrak{G}_c}\not\subset p^{\mathds{Z}}-\lambda
\]
for this particular reduction graph $\mathfrak{G}_c$.


\subsection{The case of genus $g\ge 3$}

Our main result concerns the case where the finite group $G$ is
the automorphism group $\Aut\mathfrak{G}$ of a stable graph $\mathfrak{G}$.
Remember that the action of $\Aut\mathfrak{G}$ has been extended to all of $K^N$ as outlined in Section \ref{extension}. 

\smallskip
First, we state the following observation which is valid also for $g=2$:
    \begin{prop}\label{niceAut}
Let $\mathfrak{G}$ be a stable graph of genus $g\ge 2$.
      If
      \[
\Aut\mathfrak{G}\cdot W^{\pm} = W^{\pm}
\]
then
\[
\Spec H_{\mathfrak{G}}\subset p^{\mathds{Z}}-\lambda
\]
    \end{prop}

    \begin{proof}
If $\Aut\mathfrak{G}\cdot W^{\pm}=W^{\pm}$, then the eigenvalues of $H_\sigma$ corresponding to a fixed eigenfunction are equal for all $\sigma\in \Aut\mathfrak{G}$. By our choice of $f$, it holds true that
      \[
\Spec H_\sigma\subset q^{\mathds{Z}}-\lambda
\]
Hence, the assertion follows by averaging.
\end{proof}

    In order to find necessary conditions for the spectrum of $J_{\mathfrak{G}}$ being prime powers, we need to make some definitions.

\smallskip
Let $H$ be a finite graph, and let $T$ be a spanning tree of $H$. The  tree $T^*$
obtained by replacing each edge $e$ of $H$  not in $T$ by two half-edges attached to the endpoints of $e$, is called the \emph{$*$-tree} of $T$ in $H$.

\smallskip
Let $v$ be a vertex in $H$, and $I$ a subgraph of $H$. The set of edges of $I$ is denoted by $E(I)$. Then we define
\[
\deg_I(v):=\absolute{\mathset{e\in E(I)\colon \text{$e$ is attached to $v$}}}
\]
If $K$ is sufficiently large, and $\mathfrak{G}$ is a stable graph,
then one can embed $T^*$ into the Bruhat-Tits tree for $K$, if $T$ is a spanning tree of $\mathfrak{G}$. The half-lines of $T^*$ then correspond to some elements of $\mathds{P}^1(K)$.

\bigskip
    The following example shows that the converse of Proposition \ref{niceAut} does not hold true:

    \begin{exa}\label{notNice}
      The following graph $\mathfrak{G}$ satisfies $\Aut\mathfrak{G}\cdot W^{\pm}\neq W^{\pm}$, but
      \[
      \Spec H_{\mathfrak{G}}\subset p^{\mathds{Z}}-\lambda
      \]
      Namely,
      \[
      \xymatrix{
        *\txt{$\bullet$}\ar@{-}[r]\ar@{-}@/_/[r]\ar@{-}@/^/[r]\ar@{-}@(ul,dl)&
        *\txt{$\bullet$}\ar@{-}@(dr,ur)
      }
      \]
      The reason is that it contains a $*$-tree  $T^*$ of the following shape:
      
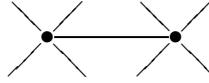
\[
      \xymatrix@=10pt{
      &&&&&\\
      &*\txt{$\bullet$}\ar@{-}[rrr]\ar@{-}[ul]\ar@{-}[dl]\ar@{-}[ur]\ar@{-}[dr]
      &&&*\txt{$\bullet$}\ar@{-}[ul]\ar@{-}[dl]\ar@{-}[ur]\ar@{-}[dr]&\\
        &&&&&
      }
      \] 
and any embedding of $T^*$ into the Bruhat-Tits tree such that the half-lines correspond to $0,1,\infty$ or points in the unit disk, will necessarily lead to 
to a half-line corresponding to a point of $K$ having absolute value $1$. The reason is that the absolute value of a point $x\in K$ depends on the distance between the vertex
$v(0,\infty,x)$  in $\mathcal{T}_K$ determined by the three points $0,\infty, x$,
and the vertex $v(0,1,\infty)$.
Namely, first assume that $0$ and $\infty$ are attached to the same vertex. In this case, all other half-lines correspond to points $x$
with $v(0,\infty,x)=v(0,1,\infty)$. Hence the absolute value of $x$ equals to one.
In the other case, $1$ and either $0$ or $\infty$ are connected to the same vertex $v$. The remaining half-line $x$ connected $v$ then has absolute value, as
again $v(0,\infty,x)=v(0,1,\infty)$.

\smallskip
Hence, it is not possible to find an automorphism $\sigma\in\Aut\mathfrak{G}$ such that
\[
\norm{\sigma x}_K<\norm{\sigma x}
\]
This proves that the eigenvalues of $H_\sigma$ corresponding to the same eigenfunction are all equal. This proves the assertion.
    \end{exa}

    \begin{dfn}
Let $H$ be a finite graph. A \emph{mouth} is a subgraph which looks like in Figure \ref{exaq} such that the three paths between the two highlighted vertices have equal positive length. These two vertices are called the \emph{corners} of the mouth.
    \end{dfn}
    
      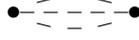
\begin{figure}
  \[
  \xymatrix@=40pt{
*\txt{$\bullet$}\ar@{--}[r]\ar@{--}@/_/[r]\ar@{--}@/^/[r]&
*\txt{$\bullet$}
  }
  \]
  \caption{A mouth-shaped graph.
        The broken lines depict paths of equal length.}\label{exaq}
  \end{figure}



\begin{cor}\label{condq}
  If the graph $\mathfrak{G}$ has no
  mouth, then
  \[
\Spec H_{\mathfrak{G}}\subset p^{\mathds{Z}}-\lambda
      \]
    \end{cor}

    \begin{proof}
As the mouth condition is equivalent to the fact that any geometric basis $W$ satisfies
      \[
\Aut\mathfrak{G}\cdot W^{\pm}=W^{\pm}
\]
the assertion follows immediately from Proposition \ref{niceAut}.
  \end{proof}


    \begin{rem}
   The mouth-shaped graph of Figure \ref{exaq}, in which the horiontal paths are all edges, 
      has an automorphism which does not permute a given basis of its fundamental group. For example, if we choose the two inner loops as such a basis $W$, then a  graph automorphism $\sigma$  can be
      defined which maps the upper inner loop to the outer loop and the lower inner loop to itself. This $\sigma$ clearly satisfies
      $\sigma W^{\pm}\neq W^{\pm}$. 
    \end{rem}

    We now can state our main result:
    
    \begin{thm}
      Let $\mathfrak{G}$ be a stable graph of genus $g\ge 3$, and let $T$ be a spanning tree.
Then $\mathfrak{G}$ has a mouth with a corner $v$ being a tip of $T$ with
  \[
\deg_{\mathfrak{G}}(v)-\deg_T(v)\le 2
\]
if and only if
\[
\Spec H_{\mathfrak{G}}\not\subset p^{\mathds{Z}}-\lambda
\]
\end{thm}

\begin{proof}
  $\Rightarrow$. Assume that $\mathfrak{G}$ has a mouth satisfying the degree condition.
  \smallskip
First observe that the two points $z_1,z_2$ from Section \ref{g=2} (the genus $2$ case), which are the solutions of equation (\ref{fpz1z2}), satisfy
  \[
\absolute{z_1+z_2}_K<1\quad\text{and}\quad\absolute{z_1z_2}_K=1
\]
as can be seen from (\ref{lineq0}).
This implies
\begin{align}\label{importanteq}
\absolute{z_1}_K=\absolute{z_2}_K
\end{align}
Further, we have
\begin{align}\label{importantineq}
\absolute{z_1-1}_K>\absolute{z_2-1}_K
\end{align}
from (\ref{condz1z2}). Again, we choose $y_2=-1$. By the mouth condition, $T^*$ contains the following subtrees:
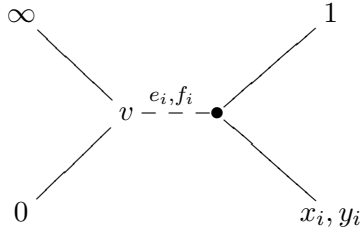
\[
\xymatrix@=25pt{
  \infty\ar@{-}[dr]&&&1\ar@{-}[dl]\\
  &v\ar@{--}[r]^{e_i,f_i}&*\txt{$\bullet$}\\
  0\ar@{-}[ur]&&&x_i,y_i\ar@{-}[ul]
}
\]
with all paths $e_i,f_i$ having a common initial edge $e$ for $i=3,\dots,g$.
We now may choose $x_i,y_i$ such that the configuration within $T^*$ looks like this:
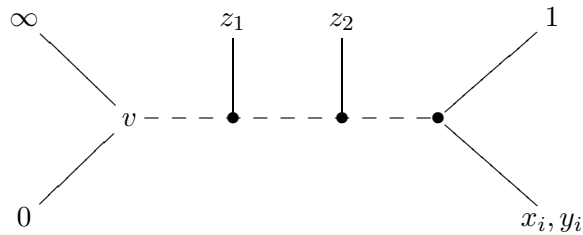
\[
\xymatrix@=25pt{
  \infty\ar@{-}[dr]&&z_1&z_2&&1\ar@{-}[dl]\\
  &v\ar@{--}[rrr]&*\txt{$\bullet$}\ar@{-}[u]&*\txt{$\bullet$}\ar@{-}[u]&*\txt{$\bullet$}\\
  0\ar@{-}[ur]&&&&&x_i,y_i\ar@{-}[ul]
}
\]
We now compute that
\[
\absolute{\beta(\xi)}_K<1
\]
for $\xi=x_i$ or $\xi=y_i$, $i=3,\dots,g$. Namely,
\begin{align*}
  \absolute{\beta(\xi)}_K&=\absolute{\frac{(z_2-1)\xi+z_1(1-z_2)}{(z_1-1)\xi+z_2(1-z_1)}}_K
  =\absolute{\frac{z_1}{z_2}}_K\absolute{\frac{1-z_2}{1-z_1}}_K<1
\end{align*}
where the second equality can be read off the tree above, and
the inequality follows from
(\ref{importanteq}) and (\ref{importantineq}).

\smallskip
We have now shown that there exist $x\in \mathcal{F}(\mathfrak{G})$ such that
\[
\norm{\sigma x}_K<\norm{x}_K
\]
Hence, there are differing eigenvalues $\lambda_{\id}$, $\lambda_\sigma$ of $H_{\id}$, $H_\sigma$, respectively, corresponding to the same eigenfunction.
This implies that their average is not of the form $p^m-\lambda$ with $m\in\mathds{Z}$. This proves the assertion.

\smallskip\noindent
$\Leftarrow$. Now, assume that $\mathfrak{G}$ does not have a mouth satisfying the degree condition. If $\mathfrak{G}$ does not have any mouth, then it holds true that
\[
\Aut\mathfrak{G}\cdot W^{\pm}=W^{\pm}
\]
for any geometric basis $W$ of $F_g$. Hence, by Proposition \ref{niceAut}, it follows that
\[
\Spec H_{\mathfrak{G}}\subset p^{\mathds{Z}}-\lambda
\]
If $\mathfrak{G}$ does have a mouth, this corresponds to a subtree of $T^*$ which contains a $*$-tree like in Example \ref{notNice}, and we conclude like in that example that
\[
\Spec H_{\mathfrak{G}}\subset p^{\mathds{Z}}-\lambda
\]
This proves the assertion.
\end{proof}

\section*{Acknowledgements}
Klaudia Oleshko is thanked for giving inspiration to pursue this work.
Wilson Z\'{u}\~{n}iga-Galindo is thanked for his willingness to proof-read
a previous version of the  manuscript and for helpful discussions and ideas, including references to the literature. Frank Herrlich is thanked for asking about the motivation for this work which lead to coming up with the idea for the main result. Vladimir Anashin is thanked for indicating a way for removing some technical condition.

\bibliographystyle{plain}
\bibliography{biblio}

\end{document}